\documentclass[12pt]{amsart}
\usepackage{amscd,amssymb}
\usepackage{graphicx}
\usepackage[arrow,matrix]{xy}

\usepackage[latin1]{inputenc}
\theoremstyle{plain}
\newtheorem{thm}[subsection]{Theorem}
\newtheorem{lem}[subsection]{Lemma}
\newtheorem{prop}[subsection]{Proposition}
\newtheorem{cor}[subsection]{Corollary}

\theoremstyle{definition}
\newtheorem{rk}[subsection]{Remark}
\newtheorem{defn}[subsection]{Definition}
\newtheorem{ex}[subsection]{Example}

\numberwithin{equation}{section} 

\newcommand{\C}{\mathbb{C}}

\newcommand{\OO}{\mathcal{O}}

\newcommand{\ir}{\mathcal{R}}
\newcommand{\ia}{\mathcal{A}}
\newcommand{\ik}{\mathcal{K}}
\newcommand{\jj}{\mathcal{J}}

\begin{document}

\title[Invariants of Topological Relative Right Equivalences]{Invariants of Topological Relative Right Equivalences}

\author[I. Ahmed]{Imran Ahmed$^{1}$}
\thanks{1 Supported by FAPESP, grant \# 2010/01895-3}
\address{Imran Ahmed, Department of Mathematics, COMSATS Institute of Information Technology, M.A. Jinnah Campus, Defence Road, off Raiwind Road Lahore, Pkistan and Departamento de Matem$\acute{a}$ticas, Instituto de Ci\^{e}ncias Matem$\acute{a}$ticas e de Computa\c{c}\~{a}o, Universidade de S\~{a}o Paulo, Avenida Trabalhador
S\~{a}ocarlense 400, S\~{a}o Carlos-S.P., Brazil.}
\email{drimranahmed@ciitlahore.edu.pk}
\author[M.A.S. Ruas]{Maria Aparecida Soares Ruas$^{2}$}
\thanks{2 Partially supported by FAPESP, grant \# 08/54222-6, and CNPq, grant
\# 303774/2008-8}
\address{Maria Aparecida Soares Ruas, Departamento de Matem$\acute{a}$ticas, Instituto de Ci\^{e}ncias Matem$\acute{a}$ticas e
de Computa\c{c}\~{a}o, Universidade de S\~{a}o Paulo, Avenida Trabalhador
S\~{a}ocarlense 400, S\~{a}o
Carlos-S.P., Brazil.} \email{maasruas@icmc.usp.br}
\author[J.N. Tomazella]{Jo\~{a}o Nivaldo Tomazella}
\address{Jo\~{a}o Nivaldo Tomazella, Departamento de Matem$\acute{a}$tica, Universidade Federal de S\~{a}o Carlos - UFSCar Caixa Postal 676 S\~{a}o Carlos - S\~{a}o Paulo, CEP 13560-905 Brazil}
\email{tomazella@dm.ufscar.br}



\maketitle


\section{Introduction}

We fix a system of local coordinates $x$ of $\C^n$. Consider the
ring $\OO_n$ of holomorphic germs $f:(\C^n,0)\to\C$ and denote by
$m_n$ its maximal ideal. Due to identification between $\OO_n$ and
the ring of convergent power series $\C\{x_1,\ldots,x_n\}$ we
identify a germ $f\in\OO_n$ with its power series $f(x)=\sum
a_{\alpha}x^{\alpha}$, where
$x^{\alpha}=x_1^{\alpha_1}\ldots x_n^{\alpha_n}$.

The Milnor number of a germ $f$ with isolated singularity, denoted by $\mu(f)$, is
algebraically defined as the $\dim_{\C}\OO_n/\jj_f$, where $\jj_f$
denotes the ideal generated by partial derivatives $\partial
f/\partial x_1$, $\ldots$, $\partial f/\partial x_n$ and the multiplicity $m(f)$ is the lowest degree in the power series expansion of $f$ at $0\in \C^n$.

A deformation $F:(\C^n\times\C,0)\to(\C,0)$, $F(x,t)=f_t(x)$, of $f$ is $\mu$-constant if
$\mu(f_t)=\mu(f)$ for small values of $t$. We denote by
$\jj_F=\langle\partial F/\partial x_1, \ldots, \partial F/\partial
x_n\rangle$, the ideal in $\OO_{n+1}$ generated by the partial
derivatives of $F$ with respect to the variables $x_1,\ldots,x_n$.

The Milnor number is a topological invariant of the singularity, more precisely if two germs of complex hypersurfaces with isolated singularities are homeomorphic, then have the same Milnor number. We also have by \cite{Le1}, for $n\neq 3$, that if a family of hypersurfaces is a $\mu$-constant family then the hypersurfaces are homeomorphic. Some definitions and properties for this number can be found in the famous book of Milnor \cite{Mi}.

The constancy of the Milnor number has several characterizations which were summarized by Greuel in the following theorem \cite{Gr}, p.161.

\begin{thm}\label{thg}(\cite{Gr}) For any deformation $F:(\C^n\times\C,0)\to (\C,0)$ of a function germ $f$ with isolated singularity the following statements are equivalent.\\
(1) $F$ is a $\mu$-constant deformation of $f$.\\
(2) For every holomorphic curve $\gamma:(\C,0)\to (\C^n\times\C,0)$
$$\nu(\frac{\partial F}{\partial t}\circ \gamma)>inf\{\nu(\frac{\partial F}{\partial x_i}\circ\gamma)\,|\,i=1,\ldots,n\},$$
(where $\nu$ denotes the usual valuation of a complex curve).\\
(3) same statement as in (2) with "$>$" replaced by "$\geq$".\\
(4) $\frac{\partial F}{\partial t}\in \overline{\jj_F}$, (where $\overline{\jj_F}$ denotes the integral closure of $\jj_F$ in $\OO_{n+1}$).\\
(5) $\frac{\partial F}{\partial t}\in\sqrt{\jj_F}$, (where$\sqrt {\jj_F}$ denotes the radical of $\jj_F$).\\
(6) The polar curve of $F$ with respect to $\{t=0\}$ does not split i.e.
$$C=\{(x,t)\in\C^n\times\C\,|\,\frac{\partial F}{\partial x_i}=0, i=1\ldots,n\}=\{0\}\times\C \,\,near \, (0,0).$$
\end{thm}

Bruce and Roberts in \cite{BR}, present the relative case, i.e., they study function germs $f:(\C^n,0)\rightarrow(\C,0)$ taking into account a germ of fixed analytic variety $(V,0)\subset(\C^n,0)$. They introduce a generalization of the Milnor number of $f$, which we will call the Bruce-Roberts number, $\mu_{BR}(V,f)$. Like the Milnor number of $f$, this number shows some properties of $f$ and $V$. For example, if we consider the group $\mathcal R_V$ of automorphisms of $(\C^n,0)$ preserving $V$ then $f$ is finitely determined with respect to the action of $\mathcal R_V$ on $\mathcal O_n$
if and only if $\mu_{BR}(V,f)$ is finite and, in this case, the codimension of the orbit of $f$ under this action is $\mu_{BR}(V,f)$.

This paper presents a study of Theorem \ref{thg} \cite{Gr} in the case of families of functions with isolated singularities defined on an analytic variety. In this relative case, not all equivalences hold. We shall denote the conditions $(1)$ to $(6)$ in the relative case by $(1_r)$ to $(6_r)$. The Example \ref{ex1} shows that the implications $(1_r)\Rightarrow(4_r)$ and $(5_r)\Rightarrow(4_r)$ do not hold in relative case. The Example \ref{ex3} proves that $(2_r)$ is not equivalent to $(3_r)$.

In Theorem \ref{th1} we show that $(2_r)\Rightarrow(3_r)\Leftrightarrow(4_r)\Rightarrow(5_r)$ in relative case. In Proposition \ref{th2} we prove the equivalence $(1_r)\Leftrightarrow(6_r)$ in relative case assuming $C=\{(x,t)\in\C^n\times\C\,|\,dF(\xi_i)=0, i=1\ldots,p\}$ is a Cohen-Macaulay variety.

In Theorem \ref{th3} we show that the implication $(4_r)\Rightarrow(1_r)$ holds in relative case assuming that the logarithmic characteristic variety of $V$, $LC(V)$, is Cohen-Macaulay such that $V$ is a hypersurface with isolated singularity.

In the last section, we consider families of quasihomogeneous functions defined on quasihomogeneous varieties based on the results of \cite{Bruna}. We discuss the topological invariance of the Bruce-Robert Milnor number. Theorem \ref{thg2} is an interesting application to the relative Zariski multiplicity conjecture. We prove that the conjecture holds for $\mu_{BR}$-constant deformations of a quasihomogeneous germ defined on quasihomogeneous varieties.

\section{Preliminary Results}

Let $f:(\C^n,0)\to(\C,0)$ be the germ of a holomorphic function with isolated singularity. Consider the analytic variety
$V\subset(\C^n,0)$. In this note we study function germs $f:(\C^n,0)\to(\C,0)$ under the equivalence relation
that preserves the analytic variety $(V,0)$. We
say that two germs $f^1,f^2:(\C^n,0)\to(\C,0)$ are $C^0$-$\ir_V$-equivalent if
there exists a germ of homeomorphism $\psi:(\C^n,0)\to(\C^n,0)$ with
$\psi(V)=V$ and $f^1\circ\psi=f^2$. That is,
$$C^0\mbox{-}\ir_V=\{\psi\in C^0\mbox{-}\ir:\psi(V)=V\}$$
where $C^0$-$\ir$ is the group of germs of homeomorphisms of $(\C^n,0)$, and we consider its action on the ring $\OO_n$ of
 germs of holomorphic functions $f:(\C^n,0)\to(\C,0)$.

We denote by $\theta_n$ the set of germs of tangent vector fields in
$(\C^n,0)$; $\theta_n$ is a free $\OO_n$ module of rank $n$. Let $I(V)$ be the ideal in $\OO_n$ consisting of germs of
analytic functions vanishing on $V$. We denote by
$\Theta_V=\{\eta\in\theta_n:\eta(I(V))\subseteq I(V)\}$, the
submodule of germs of vector fields tangent to $V$.

The tangent space to the action of the group $\ir_V$ is $T\ir_V(f)=df(\Theta_V^0)$, where $\Theta_V^0$ is the
submodule of $\Theta_V$ given by the vector fields that are zero at zero. When the point $x=0$ is a stratum in the
logarithmic stratification of the analytic variety, this is the case when $V$ has an isolated singularity at
 the origin, see \cite{BR} for details, both spaces $\Theta_V$ and $\Theta_V^0$ coincide.

In what follows we assume that $\Theta_V$ is generated by $\xi_1,\ldots,\xi_p$.

\begin{defn} Let $\jj_f(\Theta_V)$ be the ideal $\{df(\xi_i) :i=1,...,p\}$ in $\mathcal O_n$ and $f\in\mathcal O_n$. The number $$\mu_{BR}(V,f)=\dim_\C\frac{\mathcal O_n}{\jj_f(\Theta_V)}$$ is the Bruce-Roberts number of $f$ with respect to $V$.\end{defn}

We call a holomorphic map germ $F:(\C^n\times\C,0)\to (\C,0)$, $(x,t)\mapsto F(x,t)=f_t(x)$ a deformation of $f$ if $f_0=f$ and if $f_t(0)=0$ for $t$ sufficiently near to 0.

We denote by $\jj_F(\Theta_V)$ the ideal $\langle dF(\xi_i): i=1,...,p \rangle$ of $\OO_{n+1}$, where
 $dF=( \frac {\partial F}{\partial x_1},\ldots,  \frac {\partial F}{\partial x_n})$.
 The deformation $F$ of $f$  is $\mu_{BR}$-constant if $\mu_{BR}(V,f_t)=\mu_{BR}(V,f)$ for $t$ sufficiently small.

\medskip
The following is an open question in this theory.

\medskip

Is any $\mu_{BR}$-constant deformation topologically $\ir_V$ trivial?
\medskip

In \cite{BR}, Bruce and Roberts define the logarithmic characteristic variety of $V$, $LC(V)$ and show important properties of $\mu_{BR}(V,f)$ when $LC(V)$ is Cohen-Macaulay (CM). By Prop. 5.10 \cite{BR}, p.72 we know that if the codimension of $V$ is greater than one then $LC(V)$ is not CM. If $V$ is free divisor, in particular a plane curve, $LC(V)$ is CM, and, when $V$ is a quasihomogeneous hypersurface with isolated singularity, $LC(V)$ is CM by \cite{Bruna}. For any hypersurface the problem remains open.

Let us suppose that the vector fields $\xi_1,\ldots,\xi_p$ generate $\Theta_V$ for some neighbourhood $U$ at $0\in \C^n$. Then if $T_U^*\C^n$ is the restriction of the cotangent bundle of $\C^n$ in $U$, we define $LC_U(V)$ as $\{(x,\delta)\in T_U^*\C^n:\, \delta(\xi_i(x))=0,\,i=1,\ldots,p\}$. Then, $LC(V)$ is the germ of $LC_U(V)$ in $T_0^*\C^n$, and it can be shown that it is independent of the choice of generators of $\Theta_V$.

 The integral closure of an ideal $I$ in a ring $R$, is the ideal $\overline{I}$, of the elements $h\in R$ that satisfy a relation $h^k+a_1h^{k-1}+\ldots+a_{k-1}h+a_k=0$, with $a_i\in I^i$.

Teissier gave  the following characterization for the integral closure of an ideal in $\OO_n$.

\begin{thm}\label{prop1}( \cite{Te}, p.288)
If $I$ is an ideal in $\OO_n$, the following statements are equivalent.\\
1. $h\in \overline{I}$.\\
2. For each system of generators $h_1,\ldots,h_r$ of $I$ there exists a neighbourhood $U$ of 0 and a constant $c>0$ such that $$\mid h(x)\mid\leq c\sup\{\mid h_1(x)\mid,\ldots,\mid h_r(x)\mid\},\,\,\mbox{for all}\,\,x\in U.$$
3. For each analytic curve $\gamma:(\C,0)\to(\C^n,0)$, $h\circ\gamma$ lies in $(\gamma^*(I))\OO_1$.
\end{thm}

Item 3 of this theorem is called the valuation criterion since it is equivalent to the condition $\nu(h\circ\gamma)\geq \inf\{\nu(h_1\circ\gamma),\ldots,\nu(h_r\circ\gamma)\}$, where $\nu$ denotes the usual valuation of a complex curve.







\section{Main Results}

By Theorem \ref{thg}, we have the following statements for the relative case:\\
$(1_r)$ $F$ is  a $\mu_{BR}$-constant deformation of $f$;\\
$(2_r)$ For every holomorphic curve $\gamma:(\C,0)\to (\C^n\times\C,0)$
$$\nu(\frac{\partial F}{\partial t}\circ \gamma)>inf\{\nu(dF(\xi_i)\circ\gamma) : i=1,\ldots,p\};$$
$(3_r)$ Same statement as in $(2_r)$ with "$>$" ${}$ replaced by "$\geq$";\\
$(4_r)$ $\frac {\partial F}{\partial t} \in \overline {\jj_F(\Theta_V)}$;\\
$(5_r)$  $\frac {\partial F}{\partial t} \in \sqrt {\jj_F(\Theta_V)}$;\\
$(6_r)$ The polar curve of $F$ in $V$ with respect to $\{t=0\}$ does not
split, i.e.,
$$C=\left\{  (x,t) \in \C^n\times \C  :  dF(\xi_i(x)) = 0, \forall \  i= 1,\ldots , p \right\} = \{0\}\times \C \, \mbox{ near }\,  (0,0).$$

We now discuss the equivalences in the relative case. In \cite{RT}, the second and third authors investigated the $C^0$-$\ir_V$ triviality of families of functions with isolated singularity satisfying condition $(4_r)$. More precisely, in Theorem 6.4 they show that $(4_r)$ and $(6_r)$ imply $F$ is $C^0$-$\ir_V$ trivial. Moreover based on the example from \cite{RT} we show in the next example that the implications $(1_r)$ $\Rightarrow$ $(4_r)$ and $(5_r)$ $\Rightarrow$ $(4_r)$ do not hold in relative case.

\begin{ex}\label{ex1}
Let $(V,0)\subset(\C^3,0)$ be defined by $\Phi(x,y,z)=2x^2y^2+y^3-z^2+x^4y=0$ and $F:(\C^4,0)\to(\C,0)$ given by $F(x,y,z,t)=y+(a+t)x^2$. The module $\Theta_V$ is generated by $\eta_1=(2x,4y,6z)$, $\eta_2=(0,2z,x^4+4x^2y+3y^2)$, $\eta_3=(x^2+3y,-4xy,0)$ and $\eta_4=(z,0,2x^3y+2xy^2)$. The element $\frac{\partial F}{\partial t}=x^2$ is not in the integral closure of the ideal $\jj_{F}(\Theta_V)$. In fact, given $\gamma:(\C,0)\to(\C^4,0)$, $\gamma(s)=(s,-as^2,0,0)$, it follows that $\frac{\partial F}{\partial t}\circ \gamma$ is not in $(\gamma^{\ast}(\jj_{F}(\Theta_V)))_{\OO_1}$, then by Theorem \ref{prop1}, $\frac{\partial F}{\partial t}=x^2$ does not belong to $\overline{\jj_{F}(\Theta_V)}$. Now, the variety $V$ is weighted homogeneous of type $(1,2,3;6)$ and $F$ is also weighted homogeneous of degree 2 with respect to the same set of weights. Moreover, $F$ is a deformation by nonnegative weights of the $\ir_V$-finitely determined weighted homogeneous germ $f_0=y+ax^2$. Then $F$ is $C^0$-$\ir_V$-trivial (see \cite{D} and \cite{RT2}). One can also verify from direct computations that $\mu_{BR}(V,f_t)$ is constant.

Note that $\jj_F(\Theta_V)=\langle 4(a+t)x^2+4y,2z,2(a+t)x^3+6(a+1)xy-4xy\rangle$. It can be easily seen that there exists $k$ such that $m_3^k\OO_4\subset\jj_F(\Theta_V)\subset m_3\OO_4$. It implies that $m_3\OO_4\subset\sqrt{\jj_F(\Theta_V)}$. Therefore, $x^2=\frac{\partial F}{\partial t}\in\sqrt{\jj_F(\Theta_V)}$.
\end{ex}

The following example shows that $(2_r)$ is not equivalent to $(3_r)$.

 \begin{ex}\label{ex3} Let $ (V,0) \subseteq (\mathbb{C}^{2},0)$ be defined by $\Phi(x,y)=x^{3}-y^{2}=0$. The set $\Theta_{V}$ is generated by $\xi_{1}=(2x,3y),\xi_{2}=(2y,3x^{2})$. Let $F(x,y,t)=x^{5}+y^{2}+tx^5$. Thus,  $\jj_F(\Theta_V)=\left\langle  dF(\xi_1),dF(\xi_2) \right\rangle =
 \left\langle 10x^5 +6y^2 +10tx^5,10x^4y+6x^2y+10tx^4y \right\rangle $.

 We show that $\frac{\partial F}{\partial t}=x^5$ satisfies $(3_r)$ but does not satisfy $(2_r)$.
 In fact, let
  $\gamma:(\C,0)\to(\C^3,0), \,\, \gamma(s)=(\gamma_1(s),\gamma_2(s),\gamma_3(s))$,
   then $\nu(\frac{\partial F}{\partial t}\circ \gamma)=5\nu(\gamma_1)$.

   If $\nu(\gamma_2)\leq \nu(\gamma_1)$ then $\nu(dF(\xi_1)\circ \gamma)=2\nu(\gamma_2)<5\nu(\gamma_1)$.

   If $2\nu(\gamma_2)=5 \nu(\gamma_1)$ then $\nu(dF(\xi_2)\circ \gamma)=2\nu(\gamma_1)+\nu(\gamma_2)=\frac{9}{2}\nu(\gamma_1)<5\nu(\gamma_1)$.

   If $2\nu(\gamma_2)>5 \nu(\gamma_1)$ then $\nu(dF(\xi_1)\circ \gamma)=5\nu(\gamma_1)$.

   If $2\nu(\gamma_2)<5 \nu(\gamma_1)$ then $\nu(dF(\xi_1)\circ \gamma)=2\nu(\gamma_2)<5 \nu(\gamma_1)$.

    Therefore, $\nu(\frac{\partial F}{\partial t}\circ \gamma)\geq inf \{\nu(dF(\xi_1)\circ \gamma),\nu(dF(\xi_2)\circ \gamma)\}$

 Now, by taking $\alpha:(\C,0)\to(\C^3,0)$, $\alpha(s)=(s,0,0)$, we get
  $$\nu(\frac{\partial F}{\partial t} \circ \alpha)=inf \{ \nu(dF(\xi_1)\circ \alpha),\nu(dF(\xi_2)\circ\alpha)\}=5$$

\end{ex}


We establish now the following equivalences in relative case.

\begin{thm}\label{th1}
Let $V=\Phi^{-1}(0),\,\,\Phi:(\C^n,0)\to (\C^p,0)$ and $F:(\C^n\times\C,0)\to (\C,0)$ be any deformation of $f$. Then
$$(2_r) \Rightarrow (3_r) \Leftrightarrow (4_r) \Rightarrow (5_r).$$
\end{thm}

\begin{proof}
$(2_r)\Rightarrow(3_r)$ and $(4_r) \Rightarrow (5_r)$ are trivial, while $(3_r)$ $\Leftrightarrow$ $(4_r)$ is the valuation test for integral dependence, Theorem \ref{prop1}.
\end{proof}

The proof of the equivalence between $(1_r)$ and $(6_r)$ depends on the principle of conservation of numbers. The following example indicates that conditions on the variety $V$ alone are not sufficient and the principle may fail even when $V$ is a smooth variety.

\begin{ex}\label{ex2}(Example 5.9,\cite{BR})
Let $V:x_1=x_2=0$ be a non-singular surface in $\C^4$ containing $0$. Since $V$ has codimension $>1$, it follows from Proposition 5.10 in \cite{BR} that $LC(V)$ is not Cohen-Macaulay. The module $\Theta_V$ of vector fields tangent to $V$ is generated by
$$\langle x_1\frac{\partial}{\partial x_1},x_2\frac{\partial}{\partial x_2},x_2\frac{\partial}{\partial x_1},x_1\frac{\partial}{\partial x_2},\frac{\partial}{\partial x_3},\frac{\partial}{\partial x_4}\rangle.$$
Let $f(x_1,x_2,x_3,x_4)=x_1^2+x_2^2+x_3^2+x_4^2$. Then $$\jj_f(\Theta_V)=df(\Theta_V)=\langle x_1^2,x_2^2,x_1 x_2,x_3,x_4\rangle.$$ Therefore, $\mu_{BR}(V,f)=dim_{\C}\frac{\OO_4}{\jj_f(\Theta_V)}=3$. Take $F=f_t=x_1^2+x_2^2+x_3^2+x_4^2+t x_1$ a deformation of $f$. Then $dF=( 2x_1+t,2x_2,2x_3,2x_4)$. Therefore, $$\jj_F(\Theta_V)=dF(\Theta_V)=\langle 2x_1^2+tx_1,2x_2^2,2x_1 x_2+tx_2,2x_1x_2,2x_3,2x_4\rangle.$$ Note that $C:x_2=x_3=x_4=0, x_1=0\mbox{ or }x_1=-t/2$. The singularities of $f_t$, $t\neq 0$ are (0,0,0,0) and (-t/2,0,0,0). Now, $\mu_{BR}(V,f_t)\mid_{(0,0,0,0)}=1=\mu_{BR}(V,f_t)\mid_{(-t/2,0,0,0)}$.
\end{ex}

\begin{prop}\label{th2}
Let $V=\Phi^{-1}(0),\,\,\Phi:(\C^n,0)\to (\C^m,0)$ and $F:(\C^n\times\C,0)\to (\C,0)$ be any deformation of $f$. Let $C=\{(x,t)\in\C^n\times\C\,|\,dF(\xi_i)=0, i=1\ldots,p\}$ be the polar curve. Now, we state $(1_r)$ and $(6_r)$ as follows.\\
$(1_r)$ $F$ is a $\mu_{BR}$-constant deformation of $f$.\\
$(6_r)$ The polar curve of $F$ with respect to $\{t=0\}$ does not split i.e. $C=\{0\}\times\C$ near $(0,0)$. Then

(i) $(1_r)\Rightarrow(6_r)$

(ii) If the variety $C$ is Cohen-Macaulay then $(6_r)\Rightarrow(1_r)$
\end{prop}

\begin{proof}
We prove $(1_r)$ $\Rightarrow$ $(6_r)$. Choose small balls $B=\{x\in \C^n\,|\,\|x\|<\epsilon\}$, $T=\{t\in\C\,|\,|t|<\delta\}$ such that $\delta$ and $\epsilon$ are sufficiently small. Let $C=\{(x,t)\in B\times T\,|\,dF(\xi_i)=0, i=1,\ldots,p\}$, where $\xi_i$ are the generators of $\Theta_V$ and $\pi:B\times T\to T$ the projection. Then, we get
\begin{equation}\label{eq3}
\sum_{(x,t)\in C\cap B\times\{t\}}\mu_{BR}(V,f_t,x)=\mu_{BR}(V,f)
\end{equation}
for all $t\in T$. It follows that $(1_r)$ $\Rightarrow$ $(6_r)$. Conversely,  $(6_r)$ $\Rightarrow$ $(1_r)$ follows immediately from (\ref{eq3}) assuming that $C$ is Cohen-Macaulay.
\end{proof}

Note that the Cohen-Macaulay property of $LC(V)$ holds only for hypersurfaces, see \cite{BR}. We need the following lemma (see \cite{AB}, Lemma 6.1) to prove the Theorem \ref{th4}.

\begin{lem}(\cite{AB})\label{lem1}
Let $f:(\C^n,0)\to(\C^p,0)$ be an analytic map germ and $J$ an ideal of $\OO_p$. Let $f^{\ast}:\OO_p\to\OO_n$ be the ring homomorphism induced by $f$ such that $I=f^{\ast}(J)$. If $\OO_p/J$ is Cohen-Macaulay and $codim V(I)=codim V(J)$, then $\OO_n/I$ is also Cohen-Macaulay.
\end{lem}

\begin{thm}\label{th4} If $V$ is a hypersurface with $LC(V)$ CM then $(1_r)\Leftrightarrow(6_r)$.
\end{thm}
\begin{proof}  Let $\psi_F:\C^n\times\C\to
T^{\ast}(\C^n)\cong \C^{2n}$, $(x,t)\mapsto(x,dF(x))$. Let $I\subset \OO_{n+1}$ and $J\subset \OO_{2n}$ be ideals
of $C$ and $LC(V)$, respectively. Then
$\psi_F^{\ast}(J)=I$ and $(\psi_F)^{-1}(LC(V))=V(dF(\Theta_V))=C$. As $\dim LC(V)=n$ (see
\cite{BR}) and $\dim C=1$ then,  by Lemma \ref{lem1}, $C$ is CM if $LC(V)$
is CM.
\end{proof}

We now discuss the relation between equivalences in Theorem \ref{th1} and Proposition \ref{th2}. The following result was obtained by Gaffney in \cite{Ga}.

\begin{thm}\label{thga}
Let $G:\C^n\times\C^t\to(\C^m,0)$, $(z,s)\mapsto G(z,s)$, defining $X=G^{-1}(0)$ with reduced structure, $Y=0\times\C^t$ and $X_0$ the smooth part of $X$. Then $\frac{\partial G}{\partial s}\in \overline{\langle z_i\frac{\partial G}{\partial z_j}\rangle}_{\OO_X}$ for all tangent vectors $\frac{\partial}{\partial s}$ to $0\times\C^t$ iff $(X_0,Y)$ is Whitney regular.
\end{thm}

In the above theorem, $\overline{\langle z_i\frac{\partial G}{\partial z_j}\rangle}_{\OO_X}$ denotes the integral closure of the $\OO_X$-module $\langle z_i\frac{\partial G}{\partial z_j}: i,j=1,\ldots,n\rangle$. For definitions and properties, see \cite{Ga}.

Let $V$ be a sufficiently small representative of the germ of an ICIS $(V,0)$. The Milnor fiber of the complex analytic function $f$, defined on $V$, with an isolated singularity at 0 (in the stratified way), has the homotopy type of a bouquet of spheres. The Milnor number of L$\hat{e}$, denoted by $\mu_L(f)$, is defined as the number of spheres in the bouquet.

The following is a corollary of Theorem \ref{thga}.

\begin{cor}\label{cor1}
Let $V$ be an ICIS. Then $\frac{\partial F}{\partial t}\in \overline{\jj_F(\Theta_V)}$ follows that $\mu((f_t)^{-1}(0)\cap V)=\mu_L(f)$ is constant.
\end{cor}

\begin{proof}
Let $V=(\Phi)^{-1}(0)$, $\Phi:(\C^n,0)\to(\C^m,0)$ and $G:(\C^n\times\C,0)\to(\C^m\times\C,0)$ defined by $(x,t)\mapsto G(x,t)=(\Phi(x),F(x,t))$. Let $X_t=g_t^{-1}(0)$. The condition $\frac{\partial F}{\partial t}\in \overline{\jj_F(\Theta_V)}_{\OO_{n+1}}$ implies that $\frac{\partial G}{\partial t}\in \overline{\langle z_i\frac{\partial G}{\partial z_j}\rangle}_{\OO_X}$, see Lemma 6.7 in \cite{RT}. Then, from Theorem \ref{thga} we get that $(X_0,Y)$ is Whitney regular. In particular, the Milnor number $\mu(X_t)=\mu(V\cap f_t^{-1}(0))$ is constant. Therefore, $\mu_L(f_t)=\mu(V\cap f_t^{-1}(0))$ is constant.
\end{proof}

The following result was obtained by Grulha Jr., see \cite{Ni}.

\begin{thm}\label{thni}
Let $V\subset \C^n$ be a hypersurface with isolated singularity such that $LC(V)$ is Cohen-Macaulay, and $F:(\C^n\times\C^r,0)\to\C$ a family of functions with isolated singularity, then:\\
(a) $\mu_{BR}(V,f_t)$ constant for the family implies $\mu(f_t)$, $\mu_L(f_t)$ and $Eu_{f_t,V}(0)$ constant for the family, where $Eu_{f_t,V}(0)$ denotes the Euler obstruction of $f_t$ on $V$ at $0$ (see \cite{BTS,Ni}).\\
(b) When $\mu(f_t)$ is constant for the family, we have that $Eu_{f_t,V}(0)$ or $\mu_L(f_t)$ constant for the family implies $\mu_{BR}(V,f_t)$ constant for the family.
\end{thm}

We establish now the implication $(4_r)$ $\Rightarrow$ $(1_r)$ in relative case assuming that $LC(V)$ is Cohen-Macaulay such that $V$ is a hypersurface with isolated singularity.

\begin{thm}\label{th3}
Suppose $LC(V)$ is Cohen-Macaulay such that $V$ is a hypersurface with isolated singularity. Then $(4_r)$ follows $(1_r)$.
\end{thm}
\begin{proof}
Note that $\frac{\partial F}{\partial t}\in \overline{\jj_F(\Theta_V)}\subset \overline{\jj_F}$. Then $\mu(f_t)$ is constant by Theorem \ref{thg}. The Corollary \ref{cor1} implies that $\mu_L(f_t)$ is constant. Now, item (b) of Theorem \ref{thni} follows that $\mu_{BR}(f_t)$ is constant.
\end{proof}

\begin{rk}
Under the hypothesis of Theorem \ref{th3}, it follows also from Theorem 6.4 \cite{RT} that $(4_r)$ implies $C^0$-$\ir_V$ triviality of $F$.
\end{rk}

In the following example of \cite{To}, we investigate that all the conditions $(1_r)$ to $(6_r)$ are satisfied.

\begin{ex}
Let $V\subset(\C^2,0)$ defined by $\Phi(x,y)=x^3-y^2=0$ be a cusp. Consider $\Theta_V^0$ generated by $\xi_1=(2x,3y)$, $\xi_2=(2y,3x^2)$. In the $\ir_V$ classification of germs $f:\C^2\to\C$ given by \cite{BKD}, Theorem 4.9 we found the germ $F(x,y,t)=f_t(x,y)=y^2+ax^n+tx^{n+1}$, $n\geq 4$ that is finitely determined for $a\neq 0$.

Now, $df_t=(anx^{n-1}+(n+1)tx^n,2y)$. Therefore, $\jj_{f_t}(\Theta_V^0)=df_t(\Theta_V^0)=\langle 2(anx^n+t(n+1)x^{n+1})+6y^2,2(anx^{n-1}y+t(n+1)x^ny)+6x^2y\rangle$. Note that $V$ is a free divisor and its logarithmic stratification is holonomic, therefore it follows from \cite{BR}, Prop. 6.3 that $LC(V)$ is Cohen-Macaulay.

We compute $\mu_{BR}(f_t)=\dim_{\C}\frac{\OO_2}{\jj_{f_t}(\Theta_V^0)}\\
=\dim_{\C}\frac{\OO_2}{\langle 2(anx^n+(n+1)tx^{n+1})+6y^2,2(anx^{n-1}y+t(n+1)x^ny)+6x^2y\rangle}=n+3$, which is constant.

Now, let $\gamma(s)=(\gamma_1(s),\gamma_2(s),\gamma_3(s))$, then $\nu(\frac{\partial f_t}{\partial t}\circ\gamma)=(n+1)\nu(\gamma_1)$. Therefore, we have

If $n\nu(\gamma_1)<2\nu(\gamma_2)$, $\nu(df_t(\xi_1)\circ\gamma)=n\nu(\gamma_1)<(n+1)\nu(\gamma_1)$.

If $n\nu(\gamma_1)>2\nu(\gamma_2)$, $\nu(df_t(\xi_1)\circ\gamma)=2\nu(\gamma_2)<n\nu(\gamma_1)$.

If $n\nu(\gamma_1)=2\nu(\gamma_2)$, $\nu(df_t(\xi_2)\circ\gamma)=2\nu(\gamma_1)+\nu(\gamma_2)=(2+\frac{n}{2})\nu(\gamma_1)<(n+1)\nu(\gamma_1)$, $n\geq 3$.

It follows that $(2_r)$ holds, hence consequently all conditions hold.
\end{ex}

In the following example of \cite{To}, we investigate that all the conditions are satisfied except $(2_r)$.

\begin{ex}
Let $V\subset(\C^3,0)$ parameterized by $(u,-4v^3-2uv,3v^4-uv^2)$ be a swallowtail. We have $\Theta_V^0$ generated by $\xi_1=(2x,3y,4z)$, $\xi_2=(6y,-2x^2-8z,xy)$, $\xi_3=(-4x^2-16z,-8xy,y^2)$. In the $\ir_V$ classification of germs $f:\C^3\to\C$ given by \cite{BKD}, Theorem 4.10 we found the germ $F(x,y,z,t)=f_t(x,y,z)=z+ax^n+tx^{n+1}$, $n\geq 2$ that is finitely determined for $a\neq 0$ and for $n=2$ we must also have $a\neq 1/12$.

Now, $df_t=(anx^{n-1}+(n+1)tx^n,0,1)$. So, $\jj_{f_t}(\Theta_V^0)=df_t(\Theta_V^0)=\langle 2(anx^n+t(n+1)x^{n+1})+4z,6(anx^{n-1}y+t(n+1)x^ny)+xy,(-4x^2-16z)(anx^{n-1}+t(n+1)x^n)+y^2\rangle$. Note that $V$ is a free divisor and its logarithmic stratification is holonomic, so it follows from \cite{BR}, Prop. 6.3 that $LC(V)$ is Cohen-Macaulay.

We compute $\mu_{BR}(f_t)=\dim_{\C}\frac{\OO_3}{\jj_{f_t}(\Theta_V^0)}\\
=\dim_{\C}\frac{\OO_2}{\langle xy(1+6t(n+1)x^{n-1}+6anx^{n-2}),[-4x^2+8x^n(an+t(n+1)x)](anx^{n-1}+t(n+1)x^n)+y^2\rangle}$\\
$=\dim_{\C}\frac{\OO_2}{\langle xyu_1,-4x^{n+1}u_2+y^2\rangle}$, where $u_1=1+6t(n+1)x^{n-1}+6anx^{n-2}$ and $u_2=an-2anx^n(an+t(n+1)x)+t(n+1)x^3-2t(n+1)x^{n+1}(an+t(n+1)x)$ such that $u_1(0,0,t)\neq 0\neq u_2(0,0,t)$.\\
$=\dim_{\C}\frac{\OO_2}{\langle xy,4x^{n+1}+y^2\rangle}=n+3$, which is constant.

Now, using Theorem 4.5 in \cite{RT2} it follows that $x^{n+1}\in\overline{df_t(\Theta_V^0)}$, consequently $(3_r)$ holds. Taking the curve $\gamma(s)=(s,0,-\frac{an}{2}s^n,0)$ we see that $(2_r)$ does not hold.
\end{ex}

\begin{ex}
In Example \ref{ex1}, $V$ is not a free divisor, but we see that $df_t(\Theta_V)=(4(a+t)x^2+4y,2z,2(a+t)(x^2+3y)x-4xy)$. Then $C$ is a complete intersection and we can apply Prop. \ref{th2} to get $(1_r)\Leftrightarrow(6_r)$.
\end{ex}

\section{Quasihomogeneous functions and varieties}

The purpose is to prove a relative version of the following theorem due to Greuel, see \cite{Gr}.

\begin{thm}\label{thg2}
Let $f$ be a quasihomogeneous polynomial with isolated singularity and $F:(\C^n\times\C,0)\to(\C,0)$, $(x,t)\mapsto F(x,t)=f_t(x)$ a $\mu$-constant deformation of $f$. Then $m(f_t)=m(f)$ for small values of $t$.
\end{thm}

In this theorem, Greuel shows that for all $\mu$-constant deformations of a quasihomogeneous singularity the multiplicity does not change. Since constant topological type implies constant Milnor number (see \cite{Te}) he obtains, in the special cases treated in \cite{Gr}, a positive answer to Zariski's question whether for a hypersurface singularity the multiplicity is an invariant of the topological type.

In Theorem \ref{th5} we establish the relative version of above theorem. The result will follow as consequence of the following L\^e-Greuel type formula given by Nu\~no-Ballesteros, Oréfice and Tomazella in \cite{Bruna}.

\begin{thm}\label{thlg}(\cite{Bruna})
Let $(V,0)$ be a germ of hypersurface with isolated singularity defined by a weighted homogeneous function germ $\Phi:(\C^n,0)\to(\C,0)$ and let $f:(\C^n,0)\to(\C,0)$ be a $\ir_V$-finitely determined function germ. Then
$$\mu_{BR}(V,f)=\mu(f)+\mu_L(f).$$
\end{thm}

\begin{thm}\label{th5}
Let $V$ be a quasihomogeneous hypersurface in $\C^n$ with isolated singularity. Let $f:(\C^n,0)\to(\C,0)$ be a quasihomogeneous function with isolated singularity on $V$ and $F:(\C^n\times\C,0)\to (\C,0)$ a $\mu_{BR}$-constant deformation of $f$. Then $m(f_t)=m(f)$ for small values of $t$.
\end{thm}

\begin{proof}
By Theorem \ref{thlg} we have that if $\mu_{BR}(f_t)$ is constant, then $\mu(f_t)$ and $\mu_L(f_t)$ are constant by the upper semicontinuity of these numbers. Now, the result follows directly from Theorem \ref{thg2}.
\end{proof}

We give an example to show that the above result does not hold if we replace $\mu_{BR}$ by $\mu_L$, see Example V. 1. in \cite{BGG}.

\begin{ex}
Let $V=\Phi^{-1}(0)$, $\Phi(x,y,z)=x^{15}+y^{10}+z^6$ and $f_t:\C^3\to\C$, $f_t(x,y,z)=xy+tz$ are weighted homogeneous with respect to weights (2,3,5). The function $f_0=xy$ is not $\ir_V$-finitely determined, i.e., $\mu_{BR}(f_t)$ is not finite, however the curve $\Phi^{-1}(0)\cap f_0^{-1}(0)$ is an $ICIS$ in $\C^3$, see \cite{BGG}. Moreover, from \cite{BGG} it follows that $\mu(\Phi^{-1}(0)\cap f_t^{-1}(0))=\mu_L(f_t)=126$, is constant for small values of $t$ and $m(f_t)$ is not constant.
\end{ex}

As other application of the Theorem \ref{thlg} we can show the $\ir_V$-topological invariance of $\mu_{BR}$, when $V$ is a weighted homogeneous hypersurface with isolated singularity.

\begin{thm}
Let $V$ be a quasihomogeneous hypersurface in $\C^n$ with isolated singularity. Let $f,g:(\C^n,0)\to(\C,0)$ be $\ir_V$-finitely determined and $f$ $C^0$-$\ir_V$-equivalent to $g$. Then $\mu_{BR}(f)=\mu_{BR}(g)$.
\end{thm}

\begin{proof}
The hypothesis imply that $\mu_L(f)=\mu_L(g)$, this number is a topological invariant. Also, that $f$ and $g$ are $C^0$-$\ir$-equivalent, then it is well known that  $\mu(f)=\mu(g)$ \cite{Mi}. Now, we use the Theorem \ref{thlg} to get $\mu_{BR}(f)=\mu_{BR}(g)$.
\end{proof}

{\bf Acknowledgements}: The authors thank J. J. Nu$\tilde{n}$o-Ballesteros for useful comments. The first author is thankful to FAPESP for its support and the second author is thankful to FAPESP and CNPq for their partial supports in producing this manuscript.





\end{document}